\documentclass[a4paper,10pt]{article}
\usepackage{stmaryrd}
\usepackage{amsfonts}
\usepackage{bbm}
\usepackage{amscd}
\usepackage{mathrsfs}
\usepackage{latexsym,amssymb,amsmath,amscd,amscd,amsthm,amsxtra}
\usepackage[dvips]{graphicx}
\usepackage[utf8]{inputenc}
\usepackage[T1]{fontenc}
\usepackage{lmodern}
\usepackage{amssymb}
\usepackage[all]{xy}
\usepackage{nicefrac,mathtools,enumitem}
\usepackage{microtype}
\usepackage{CJK}
\usepackage{amssymb}
\usepackage{epstopdf}
\usepackage{graphicx}
\usepackage{graphics}
\usepackage[T1]{fontenc}
\usepackage[utf8]{inputenc}
\usepackage{authblk}
\usepackage{subfigure}
\usepackage{caption}
\usepackage{framed}
\usepackage{mathrsfs}
\usepackage[noend]{algpseudocode}
\usepackage{algorithmicx,algorithm}
\usepackage{enumitem}
\usepackage{graphicx}
\usepackage{color}
\usepackage{multirow}
\usepackage{diagbox}
\usepackage{algorithmicx,algorithm}
\usepackage{framed}
\usepackage{booktabs}
\usepackage{threeparttable}
\usepackage{color}

\numberwithin{equation}{section}
\newtheorem{thm}{Theorem}
\newtheorem{lem}{Lemma}
\newtheorem{cor}{Corollary}
\newtheorem{pro}{Proposition}

\newtheorem{rmk}{Remark}

\newtheorem{ass}{Assumption}

\newcommand {\emptycomment}[1]{}

\newcommand{\be }{\begin{equation}}
\newcommand{\ee }{\end{equation}}

\newcommand{\huaB}{\mathcal{B}}

\newcommand{\huaF}{\mathcal{F}}
\newcommand{\huaG}{\mathcal{G}}

\newcommand{\huaV}{\mathcal{V}}

\newcommand{\huaX}{\mathcal{X}}

\newcommand{\huaK}{\mathcal{K}}

\newcommand{\g}{\mathfrak g}

\newcommand{\nono}{\nonumber}

\newcommand{\e}{\mathbbm{e}}

\newcommand{\w}{\widetilde}

\newcommand{\f}{\frac}
\newcommand{\ttt}{\theta}
\newcommand{\aaa}{\alpha}

\newcommand{\nn}{\langle}
\newcommand{\mm}{\rangle}
\newcommand{\ww}{\widetilde}

\newcommand{\mbb}{\mathbb}
\newcommand{\eee}{\epsilon}
\newcommand{\ooo}{\omega}

\def\w{\widetilde}

\def\bea{\begin{eqnarray}}
	\def\eea{\end{eqnarray}}
\def\be{\begin{equation}}
	\def\ee{\end{equation}}
\def\blm{\begin{lem}}
	\def\elm{\end{lem}}
\def\btm{\begin{theorem}}
	\def\etm{\end{theorem}}

\def\ff{\Phi}
\def\e{\mathcal{E}}
\def\nb{\nabla}

\def\g{\gamma}

\def\bea{\begin{eqnarray}}
	\def\eea{\end{eqnarray}}
\def\be{\begin{equation}}
	\def\ee{\end{equation}}
\def\blm{\begin{lem}}
	\def\elm{\end{lem}}

\def\kk{\kappa}

\def\bea{\begin{eqnarray}}
\def\eea{\end{eqnarray}}
\def\be{\begin{equation}}
\def\ee{\end{equation}}
\def\blm{\begin{lem}}
\def\elm{\end{lem}}

\def\bs{\begin{small}}
\def\es{\end{small}}

\def\kk{\kappa}

\begin{document}
\title{
Distributed Stochastic Constrained Composite Optimization over Time-Varying Network with a Class of Communication Noise\thanks{ Zhan Yu, Daniel W. C. Ho and Jie Liu are with the Department of Mathematics,
	City University of Hong Kong, Kowloon, Hong Kong (e-mail: zhanyu2-c@my.cityu.edu.hk, mathyuzhan@gmail.com; madaniel@cityu.edu.hk; jliu285-c@my.cityu.edu.hk), Deming Yuan is with the School of Automation, Nanjing University of Science and Technology, Nanjing, China (e-mail: dmyuan1012@gmail.com)}
 }\vspace{2mm}
\author{ Zhan Yu, Daniel W. C. Ho, Deming Yuan, Jie Liu}
\date{}



\maketitle

\begin{abstract}
	This paper is concerned with  distributed stochastic multi-agent constrained optimization problem over  time-varying network with a class of communication noise.  This paper considers the problem in composite optimization setting which is more general in the literature of noisy network optimization. It is noteworthy that the mainstream existing methods for noisy network optimization are Euclidean projection based. Based on Bregman projection-based mirror descent scheme, we present a non-Euclidean  method and investigate their convergence behavior. This method is the distributed stochastic composite mirror descent type method (DSCMD-N) which provides a more general algorithm framework. Some new error bounds for DSCMD-N are obtained.
	To the best of our knowledge, this is the first work to analyze and derive convergence rates of optimization algorithm in noisy network optimization. We also show that an optimal rate of $O(1/\sqrt{T})$ in nonsmooth convex optimization  can be obtained for the proposed method under appropriate communication noise condition. Moveover,  novel convergence results  are comprehensively derived in  expectation convergence, high probability convergence and almost surely sense.
\end{abstract}

\section{Introduction}
In recent years,  distributed consensus control and optimization problems over networked system are studied extensively ( \cite{s1, 2016, direct1, n2, ut, sto, dmcb1, dmcb2, ljy, duchi, n3, py, s2, smart, discom, timev2, wj, hong2, ccc, hong, yzhang}). These problems arise in a variety of application domains, such as localization in sensor networks (e.g. \cite{sensor}, \cite{yzhang}), smart grid (e.g. \cite{smart}), utility maximization (e.g. \cite{ut}), allocation of resources in microeconomics (e.g. \cite{micro}). On the other hand, there always exists noise in realistic scenario, intrinsic disturbances with different types of noises often appear in many multi-agent networked systems. For networked systems with active communication channels, besides the inherent disturbance, probably one of the most important issues on disturbance is the communication noise. The communication noise is unavoidable in signal transmission and information communication process. Recently, there have been many works circling around the effects of noise among nodes of the networks on control or optimization methods (e.g. \cite{adv1, n1, n2, noiau, jq, zd1, fl, n3, noiz, zd2, zd4, zd3}). These works focus on different classes of noises, including some common classes such as  bounded noise, decay noise, mean-zero noise.

In this paper, we mainly consider the minimization of a sum of locally known convex functions that are distributed over a network with a class of communication noise. In this model,  each agent has its own associated (perhaps nonsmooth) objective function (e.g. \cite{s1, direct1, n2, sto, s2}). For solving this kind of problem, a variety of methods have emerged recently. In these methods, distributed optimization method has been shown to be one of the most powerful methods for its advantage of saving energy and reducing unnecessary waste of resources. Recent years have witnessed progress of distributed optimization in numerous aspects. Modern studies on distributed optimization start from the classical distributed subgradient method (\cite{s1}). The seminal research \cite{s1} is inspired by a deterministic gradient descent model over network system. Also, their work treats with unconstrained decision variable. Consequently, studies on distributed stochastic subgradient method appears (\cite{s2}). Disturbance on subgradient is considered to capture the dynamical environment in real world. In \cite{2016}, a distributed gradient-push method is established without requiring information of either the number of agents or the graph sequence. In the same period, the work \cite{direct1} provides a novel distributed method to better capture the direct structure of the network topology, the convergence relies on a core matrix analysis result in \cite{kc}. In a different way, the work in \cite{duchi} presents a (stochastic) dual averaging-based method, the method is based on maintaining and forming weighted averages of subgradients throughout the network. In what follows, several works which improve \cite{duchi} appear (e.g. \cite{duallee, duals}).
On the other hand, several works also turn to investigate the case when local objective functions are nonconvex (e.g. \cite{ccc}). Mirror descent technique has been utilized in distributed optimization domain recently (e.g. \cite{lmd}), one of the main features of mirror descent is that it can better reflect the geometry of underlying space. Moreover, online distributed optimization has also become a new direction recently (e.g. \cite{yuanmd, ljy, timev2}), online distributed methods are often investigated to handle the dynamical environment of local objective functions. Based on the sum structure of the global objective function, a great deal of works aiming at solving the problem are consensus-based.  The realization of consensus is an essentially necessary condition for the convergence of these methods.

The main goal of this paper is to study  distributed optimization problems by addressing following considerations: (i) Since uncertain stochastic disturbances always exist in real life environment, it is desirable to consider the topic of solving distributed optimization problem over network in which some class of communication noises exist among nodes; (ii) The existing optimization methods over noisy network are Euclidean gradient projection based (see e.g. \cite{n2,n3}), is it possible to consider some more general frameworks and provide general methods to solve them in some class of noisy network? (iii) Although under suitable conditions, in the setting of distributed optimization when communication noise exists over network, almost surely existence result of the optimal solution is proven for gradient descent-based methods in main existing works such as \cite{n2}, \cite{n3}. Explicit description of convergence rate is still absent in this literature. Is it possible to derive convergence rate results in settings when some class of communication noise exists? Also, the optimization methods are very poorly explored over noisy network, it is desirable to develop some algorithms for such optimization problem. To this end, we consider multi-agent composite optimization problem over time-varying noisy network in this paper. Specifically, we analyze the following problem:
\be
\mathrm{minimize}_{x\in\huaX} \ \ F(x)=\sum_{i=1}^NF_i(x):=\sum_{i=1}^N\big[f_i(x)+\chi_i(x)\big], \ \  \label{problem1}
\ee
where $\huaX$ is non-empty convex constraint set and each local cost function $f_i$ (only known to node $i$) is convex and maybe nonsmooth. $\chi_i$ is a simple convex regularization function associated with node $i$.
In recent years, there are a few works on distributed methods treating with the aforementioned problems with composite framework (e.g.  \cite{yuanmd}). \cite{yuanmd} mainly focuses on online optimization and develops an online two-point bandit feedback mirror decent based method. \cite{discom} analyzes the distributed optimization problem over relay-assisted networks. From a different perspective, the work in this paper considers the network with communication noise and attempts to develop distributed optimization methods that are suitable to noisy network. Meanwhile, we study the convergence of the distributed composite optimization methods in noisy circumstance.

In this paper, inspired by stochastic approximation theory in \cite{n2, n3}, we develop a class of stochastic optimization method for solving above composite optimization problem. The problems are considered  over time-varying network that has a class of communication noise effects among nodes in information transmission process. We propose distributed stochastic composite mirror descent (DSCMD-N method) for Problem (\eqref{problem1}). The convergence results are analyzed in detail. Specifically, we are interested in the convergence behavior of the methods under different selections of  stepsizes. The expected convergence bound and high probability convergence bound are established respectively. In what follows, the discussion on selection of stepsizes and corresponding convergence rates are provided. Note that, by taking composite regularization function into consideration, this work also extends the former works in same literature to a more general setting. For the proposed DSCMD-N, by implementing Bregman divergence instead of former Euclidean distance in works such as \cite{n2}, \cite{n3}, the DSCMD-N method extends the projection structure of these methods to more general setting. Explicit rate $O(\f{1}{\sqrt{T}})$ result is obtained for expected function error for DSCMD-N method under appropriate selection of stepsize.  The convergence behavior is described by convergence bound in terms of $\aaa_t$ (stepsize for stochastic gradient) and $r_t$ (decaying rate for noise vector) and some cross terms of them. The error bound obtained in this work describes some intrinsic trade-off between  $\{\aaa_t\}$ and  $\{r_t\}$.

The technical contributions of this paper can be summarized as follows:

(1) New method DSCMD-N is presented for distributed optimization over a class of noisy network. Existing works in the same literature such as \cite{n2}, \cite{n3} are all Euclidean projection based. By presenting DSCMD-N method, we extend these former works to a more general setting in the proposed network model. In contrast to previous work in the same literature (e.g. \cite{n3}), since the Bregman divergence is utilized, the underlying geometry structure of distributed optimization problem is better reflected. The flexible selection of mirror map (distance generating function) can enable us to generate efficient updates to face the noisy network optimization. As special cases, when we take distance generating function as the norm squared function  $\f{1}{2}\|\cdot\|^2$, the entropy induced function $\sum_{i=1}^n[x]_i\ln[x]_i$ or $l_p$ norm squared function $\f{1}{2}\|\cdot\|_p^2$, and take $\chi_i$ as some specific regularizers, the DSCMD-N method can include a wide range of algorithm class that previous works on noisy network optimization does not consider. In addition, the intrinsic results among the network noise, non-Euclidean structure and the composite terms are provided (Theorem \ref{thm1}).

(2) The convergence behaviors for DSCMD-N  are comprehensively investigated. We obtain two types of error bounds for expected error: expected bound and high probability bound. The bounds are in terms of some cross terms consisting of  stepsizes $\{\aaa_t\}$ and communication noise decaying rate $\{r_t\}$. The stepsize selection is comprehensively conducted under effects of $\{r_t\}$ in different orders. The corresponding convergence rates are obtained. To the best of our knowledge, all these rates are first achieved in the setting of optimization over noisy network. We also show that the optimal expected rate and high probability of $O(\f{1}{\sqrt{T}})$ can be obtained under some conditions on $\{\aaa_t\}$ and $\{r_t\}$. Furthermore, a new almost sure convergence type result is derived for the local sequence. This result is new in the literature of distributed optimization.

(3) Composite optimization is investigated in a class of time-varying noisy network. Hence, for different purposes on some concrete distributed optimization problems, flexible selection of regularization terms becomes possible. By  taking appropriate regularization terms into consideration, the proposed methods are potentially flexible to reflect certain structure features of the solution of distributed optimization problem.
This work allows the objective functions to be nonsmooth.
This fact makes the proposed methods more flexible to handle optimization problems when tough smoothness conditions are added on objective functions. Also, the methods are convenient for optimization over a class of time-varying network, in contrast to static network.

\textbf{Notation and terminology:}
Denote the $n$-dimension Euclidean space by $\mathbb R^n$, and the set of positive real numbers by $\mathbb R^{+}$. For a vector $v\in \mathbb R^n$, use $\|v\|$, $[v]_i$ to denote its Euclidean $2$-norm and its $i$th entry. Use $\|v\|_1$ to denote $1$-norm $|[v]_1|+|[v]_2|+\cdots+|[v]_n|$. The inner  product of two vectors $v_1$, $v_2$ is denoted by $\nn v_1, v_2\mm$. For a matrix $M\in\mathbb R^{n\times n}$, denote the element in $i$th row and $j$th column by $[M]_{ij}$. Use $I_n$ to denote the identity matrix.
A function $f$ is $\sigma_{f}$-strongly convex over domain $\huaX$ if for any $x,y\in \huaX$ and $z\in [0,1]$, $f(z x+(1-z)y)\leq z f(x)+(1-z)f(y)-\f{\sigma_{f}z(1-z)}{2}\|x-y\|^2$. Denote the gradient operator by $\nabla$, when $f$ is differentiable, the $\sigma_f$-strongly convex inequality above is equivalent to $f(x)\geq f(y)+\nn\nabla f(y),x-y\mm+\f{\sigma_{f}}{2}\|x-y\|^2$.
For two functions $f$ and $g$,  write $f(n)= O(g(n))$ if there exist $N<\infty$ and positive constant $C<\infty$ such that $f(n)\leq Cg(n)$ for $n\geq N$. For a random variable $X$, use $\mathbb E[X]$ to denote its expected value.

\section{Problem setting and preliminaries}

Let $\huaG^t=(\huaV,E^t,P^t)$ be a directed graph which denotes the information communication among the nodes at time $t$. $\huaV=(1,2,...,N)$ is the node set. $E^t=\{(i,j)|[P^t]_{ij}>0,i,j\in\huaV\}$ is the set of active links with $P^t$ being the weight matrix at time $t$. $(i,j)\in E^t$ corresponds to the case when agent $i$ and agent $j$ have information communication at time $t$.

The objective of the paper is to cooperatively solve the composite optimization Problem (\eqref{problem1})  through communication among the agents of a multi-agent system described by graph $\huaG^t$ in a constrained setting. The decision space $\huaX\subseteq \mbb R^n$ for the state variable $x$ is a convex and compact set. Recall that a compact constrained condition on $\huaX$ is standard and commonly considered in works on mirror descent type methods (e.g. \cite{md}, \cite{lmd}). For agent $i\in\huaV$, we assume that there is a corresponding local cost function $f_i$. $f_i$ is assumed to be convex and perhaps  nonsmooth. We assume that the set of nonempty optimal solution of the problems considered in this paper is denoted by $\huaX^*$ with optimal value $f(x^*)$ for any $x^*\in\huaX^*$.  The following standard assumption is made on the graph $\huaG^t$.
\begin{ass}\label{graph}
	\textnormal{The communication matrix $P^t$ is doubly stochastic. $i.e.$, $\sum_{i=1}^N[P^t]_{ij}=1$ and $\sum_{j=1}^N[P^t]_{ij}=1$ for any $i,j\in \huaV$. There exists some positive integer $B$ such that the graph $(\huaV,\cup_{t=sB+1}^{(s+1)B}E^t)$ is strongly connected for every $s\geq0$. There exists a scalar $0<\theta<1$ such that $[P^t]_{ii}\geq\theta$ for all $i\in\huaV$ and $t$, and $[P^t]_{ij}\geq\theta$ if $(j,i)\in E^t$.}
\end{ass}

The network model in Assumption 1 is widely used in
distributed multi-agent optimization community (e.g., \cite{s1}, \cite{sto}). In this paper,  $P(t,s)=P^tP^{t-1}\cdots P^s$ is used to denote transition matrix when $t\geq s$; the notation $P(t,t+1)=I_n$ is also used. The following consequence in \cite{s1} is basic for the analysis over multi-agent time-varying network.

\begin{lem}\label{Nedic}
	Under Assumption \ref{graph}, for all $i,j\in\huaV$ and $t,s$ satisfying $t\geq s\geq1$, we have
	\bea
	\nono\big|[P(t,s)]_{ij}-\f{1}{N}\big|\leq\ooo\gamma^{t-s},
	\eea
	in which
	\bea
	\nono&&\ooo=(1-\f{\ttt}{4N^2})^{-2},\\
	\nono&& \gamma=(1-\f{\ttt}{4N^2})^{\f{1}{B}}.
	\eea
\end{lem}

The stochastic methods in this paper is first-order stochastic approximation based. We make some assumptions on subgradients of the objective functions. We assume that the nodes can only compute the noisy subgradients of its corresponding objective functions. In what follows, we use $\huaF_t$ to denote the $\sigma$-algebra of the history up to time $t$. In this paper, we assume that all random processes are adapted to the filtration $\huaF_t$. The following assumption on stochastic gradient is standard in stochastic constrained convex optimization studies (\cite{md}, \cite{sto}).
\begin{ass}\label{tidu}
	\textnormal{At any point $x\in\huaX$, let the stochastic subgradient $\ww{g}_i(x)$ be such that $\mbb E[\ww{g}_i(x)|\huaF_{t-1}]=g_i(x)\in\partial f_i(x)$ and $\mbb E[\|\ww{g}_i(x)\|^2|\huaF_{t-1}]\leq G_f^2$.}
\end{ass}

This paper focuses on the network optimization with communication noise. We assume that the noise exists over the network among agents. The noise needs to be considered in information communication process of state variables of agents. In our model, we consider the following type of noises among agents:  the communication noise between node $i$ and node $j$ at instance $t$ denoted by $\{n_{ij}^t\}$ with $n_{ij}^t=r_t\xi_{ij}^t$ with a noise magnitude decaying rate $\{r_t\}$ is assumed. For the random variable $\xi_{ij}$, we assume that the following assumption holds.

\begin{ass}\label{noise}
	\textnormal{At any time instance $t$, the noise on link $(i,j)$ is independent of the noise on link $(i',j')$ for $i\neq i'$, $j\neq j'$. The communication noise $\{r_t\xi_{ij}^t\}$, $i,j\in \huaV$ over the time-varying network is a random sequence with $\mbb E[\|\xi_{ij}^t\|^2|\huaF_{t-1}]\leq\nu$. }
\end{ass}

\begin{rmk}
	In stochastic distributed consensus control or optimization with noisy links, a common assumption is that the noise sequence $\{n_{ij}^t\}$ satisfies $\sum_{t=1}^{\infty}\mbb E[\|n_{ij}^t\|^2|\huaF_{t-1}]<\infty$ and has zero means and finite variances (e.g. \cite{noiau}, \cite{noiz}). This assumption and Borel-Cantelli lemma imply that the magnitude of the communication noises decays to zero:
	\be
	\nono \|n_{ij}^t\|\rightarrow0, \ a.s.
	\ee
	In this paper, in order to investigate some novel explicit convergence rates under theoretical framework of rate analysis, we assume that the noise $\{n_{ij}^t\}$ satisfies Assumption \ref{noise} with  magnitude decaying rate $r_t$. In fact, the noise shares a similar noise decaying feature with \cite{noiau}, \cite{noiac}:
	\be
	\nono \mbb E[\|n_{ij}^t\|]\rightarrow0.
	\ee
	Assumptions in this paper can include some types of noises that the assumption $\sum_{t=1}^{\infty}\mbb E[\|n_{ij}^t\|^2|\huaF_{t-1}]<\infty$ does not cover. For example, the noise $n_{ij}^t=r_t\xi_{ij}^t$ with `` $r_t\geq O(1/\sqrt{t})$ and all $\xi_{ij}^t$ have equal non-zero $\mbb E[\|\xi_{ij}^t\|^2]$'' satisfies assumptions in this paper, but does not satisfy assumption $\sum_{t=1}^{\infty}\mbb E[\|n_{ij}^t\|^2|\huaF_{t-1}]<\infty$ as in \cite{noiau}, since $\sum_{t=1}^{\infty}r_t=\infty$ when $r_t\geq O(1/\sqrt{t})$. The fact shows that our work can cover some noise types that the work under $\sum_{t=1}^{\infty}\mbb E[\|n_{ij}^t\|^2|\huaF_{t-1}]<\infty$ assumption can not handle. In practice, as stated in \cite{adver}, \cite{adv1}, \cite{noiau} and simulation examples of \cite{noiz}, there exists such type of noise model in common network engineering, for example, some injected false data and noise signals decay with time in some network adversarial attack, or the communication noise decays when a multi-agent system moves away gradually from noisy source, or the noise is caused by a damped external noise source.
\end{rmk}

The following Azuma-Hoeffding lemma (\cite{azum1}, \cite{azum2}) is needed to derive high probability bound and rate of DSCMD-N later.
\begin{lem}\label{azuma}
	Let $\{X_t\}$ be a martingale difference sequence satisfying $|X_t|\leq\tau_t$, then  for any $\eee>0$,
	\be
	\nono\mathrm{Prob}\big(\sum_{t=1}^TX_t\geq\eee\big)\leq \exp\big(-\f{\eee^2}{2\sum_{t=1}^T\tau_t^2}\big).
	\ee
\end{lem}

In optimization literature, mirror descent is a powerful extension of classical gradient descent. Generally, in contrast to gradient descent, for a given decision space defined on a Hilbert space, the mirror descent can relax the Hilbert space structure and employ a mirror map $\Phi:\huaK\rightarrow \mbb R$ to better reflect the geometric properties of the decision variables from some Banach space $\huaK$. In this paper, we will consider $\huaK=\mbb R^n$ endowed with a norm $\|\cdot\|$ which may be a non-Euclidean norm, that can better reflect the non-Euclidean geometric structures of decision variable from $\mbb R^n$. To introduce the basic distributed mirror descent scheme, we consider a continuously differentiable $\sigma_{\ff}$-strongly convex mirror map (\emph{distance generating} function) $\ff:\mbb R^n\rightarrow \mbb R$, define the \emph{Bregman divergence} associated with $\ff$ as
\bea
\nono D_{\ff}(x,y)=\ff(x)-\ff(y)-\nn\nb\ff(y),x-y\mm.
\eea
In Section III, for the Bregman divergence, we need the following assumption. The assumption is standard in investigations of mirror descent type methods (\cite{md}, \cite{lmd}).

\begin{ass}
	\textnormal{We assume  that the mirror map $\ff$ is chosen such that $\|\nb\ff(x)-\nb\ff(y)\|\leq L_{\ff}\|x-y\|$ for any $x,y\in\Omega$ for some $L_{\ff}$. For any vectors $a$ and $\{b_i\}_{i=1}^N$ in $\mbb R^n$, the Bregman divergence satisfies the separate convexity in the following sense: $D_{\ff}(a,\sum_{i=1}^N\nu_ib_i)\leq\sum_{i=1}^N\nu_iD_{\ff}(a,b_i)$, $\nu_i\in [0,1]$ and $\sum_{i=1}^N\nu_i=1$.}
\end{ass}

\section{DSCMD-N Algorithm: main convergence results}
In this section, we consider Problem (\eqref{problem1}), minimizing $F(x)=\sum_{i=1}^NF_i(x):=\sum_{i=1}^N\big[f_i(x)+\chi_i(x)\big]$ over noisy network. We solve the problem by providing a distributed stochastic composite mirror descent method which we call it the \textbf{DSCMD-N} method. In the algorithm, for each $i\in\huaV$, the local variable $x_i^t$ evolves as follows
\begin{eqnarray}
	y_{i}^t&=&\sum_{j=1}^N[P^t]_{ij}(x_{j}^t+r_t\xi_{ij}^t), \label{md1}\\
	x_i^{t+1}&=&\arg\min_{x\in \huaX}\{\nn \ww{g}_i^t,x\mm+\f{1}{\aaa_t}D_{\ff}(x,y_i^t)+\chi_i(x)\},\label{md2}
\end{eqnarray}
where $[P^t]_{ij}$, $i,j\in\huaV$ denotes the elements of communication weight matrix $P^t$ satisfying the conditions in Assumption \ref{graph}. It denotes the weight assigned by node $i$ to the estimate coming from node $j$. In the algorithm we are concerned with the case when communication links are noisy with noise assumptions in Assumption \ref{noise}. Therefore, the node $i$ has only access to a noise corrupted value of its neighbor's local decision variable (noisy observation). (\eqref{md1}) describes the noisy information communication process between $i$ and its neighbours. Then, query the stochastic subgradient oracle at $y_{i}^t$ to get a stochastic subgradient $\ww{g}_i^t:=\ww{g}_i(y_i^t)$, such that $\mbb E[\ww{g}_i^t|\huaF_{t-1}]=g_i(y_i^t)\in\partial f_i(y_i^t)$ is a subgradient of $f_i$ at $y_i^t$. In (\eqref{md2}), we perform a Bregman projection for variable $y_i$ to decision space $\huaX$ to get variable $x_i^{t+1}$. A composite mirror descent scheme is considered in this Bregman projection with stepsize $\aaa_t$ and composite term $\chi_i(x)$. We remark that the simple composite function $\chi_i$ associated with node $i$ can be different from each other. Here, $\chi_i(x)$, $i\in\huaV$ are supposed to be some simple convex regularization function with supremum subgradient $G_{\chi_i}$. In this section $\partial\chi_i(x)$ is used to denote the \emph{subdifferential set} of $\chi_i$ at $x\in\huaX$. For the subgradients of $\chi_i$. We denote
\be
\nono G_{\chi_i}=\sup_{g\in\cup_{x\in\huaX}\partial\chi_i(x)}\|g\|
\ee
and
\be
\nono G_{\chi}=\sup_{i\in \huaV}G_{\chi_i}.
\ee
In this section, we also denote
\be
\nono D_{\ff,\huaX}^2=\sup_{x,y\in\huaX}D_{\ff}(x,y).
\ee
In fact, the finiteness of $G_{\chi}$ and $D_{\ff,\huaX}$ follows from the compactness of $\huaX$, the strong convexity of $\Phi$ implies $\sup_{x,y\in\huaX}\|x-y\|\leq\sqrt{\f{2}{\sigma_{\ff}}}D_{\ff,\huaX}$. In this paper, it is unnecessary to know the concrete value of $G_{\chi}$ and $D_{\ff,\huaX}$. The finiteness of them is enough to provide the rigorous convergence analysis of the algorithm. To investigate the convergence behavior of DSCMD-N, we denote the Bregman projection error by
\bea
\nono \e_i^t=x_i^{t+1}-y_i^t.
\eea
We start with the following error estimate on $\e_i^t$.

\begin{lem}\label{error}
	The Bregman projection error satisfies
	\be
	\nono\mbb E[\|\e_i^t\|]\leq\f{G_f+G_{\chi}}{\sigma_{\ff}}\aaa_t.
	\ee
\end{lem}
\begin{proof}
	According to the first-order optimality condition, there exists $h_i^{t+1}\in\partial\chi_i(x_i^{t+1})$ such that
	\be
	\nono\nn\aaa_t \ww{g}_i^t+\nabla\ff(x_i^{t+1})-\nabla\ff(y_i^t)+\aaa_t h_i^{t+1},x-x_i^{t+1}\mm\geq0, \ \forall x\in\huaX.
	\ee
	Setting $x=y_i^t$ in above inequality, we obtain that
	\be
	\nono\nn\aaa_t \ww{g}_i^t+\nabla\ff(x_i^{t+1})-\nabla\ff(y_i^t)+\aaa_t h_i^{t+1},y_i^t-x_i^{t+1}\mm\geq0.
	\ee
	The above inequality implies that
	\be
	\nono\nn\aaa_t(\ww{g}_i^t+h_i^{t+1}),y_i^t-x_i^{t+1}\mm\geq\nn\nb\ff(y_t^t)-\nb\ff(x_i^{t+1}),y_i^t-x_i^{t+1}\mm.
	\ee
	Use Cauchy inequality to the left hand side and $\sigma_{\ff}$-strong convexity of $\ff$ to the right hand side of above inequality, it can be obtained that
	\be
	\nono\aaa_t(\|\ww{g}_i^t\|+G_{\chi})\|y_i^t-x_i^{t+1}\|\geq\sigma_{\ff}\|y_i^t-x_i^{t+1}\|^2.
	\ee
	Eliminate same term $\|y_i^t-x_i^{t+1}\|^2$ on both sides and take conditional expectation on $\huaF_{t-1}$, we have
	\be
	\mbb E\big[\|\e_i^t\|\big|\huaF_{t-1}\big]\leq\f{1}{\sigma_{\ff}}\big(\mbb E\big[\|\ww{g}_i^t\|\big|\huaF_{t-1}\big]+G_{\chi}\big)\aaa_t.
	\ee
	The desired result is obtained after taking total expectation of above inequality on both sides.
\end{proof}

We are ready to give the following disagreement result which is necessary to establish the main convergence result of this section. In what follows, for nodes with estimates $x_i^t$, $i\in \huaV$, we denote the average estimate of them at time $t$ by
\bea
\nono\bar{x}^t=\f{1}{N}\sum_{i=1}^Nx_i^t.
\eea
\begin{lem}\label{disagreement}
	Under Assumptions \ref{graph}-\ref{noise}, let $\{x_i^t\}$ be the sequences in DSCMD-N. Then for any $j\in\huaV$,
	\bea
	\nono&&\sum_{t=1}^T\sum_{i=1}^N\mbb E\big[\|x_i^t-x_j^t\|\big]\leq\f{2N\ooo}{1-\g}\|x_j^0\|\\
	&&+\big(4N+\f{2N^2\ooo}{1-\g}\big)\sum_{t=0}^T\Big[\f{G_f+G_{\chi}}{\sigma_{\ff}}\aaa_t+N\sqrt{\nu}r_t\Big].
	\eea
\end{lem}
\begin{proof}
	For $\forall i\in \huaV$, set $\xi_i^t=\sum_{j=1}^N[P^t]_{ij}\xi_{ij}^t$, by iterating recursively, it can be obtained that
	\bea
	\nono&&x_i^t=\sum_{s=1}^t\sum_{j=1}^N[P(t-1,s)]_{ij}(\e_j^{s-1}+r_{s-1}\xi_j^{s-1})\\
	\nono&& \ \ \ \ \ \  +\sum_{j=1}^N[P(t-1,0)]_{ij}x_j^0\\
	\nono&&\bar{x}^t=\f{1}{N}\sum_{s=1}^t\sum_{j=1}^N(\e_j^{s-1}+r_{s-1}\xi_j^{s-1})+\f{1}{N}\sum_{j=1}^Nx_j^0.
	\eea
	Then it follows that
	\bea
	\nono&&\|x_i^t-\bar{x}^t\|\leq\sum_{j=1}^N\big|[P(t-1,0)]_{ij}-\f{1}{N}\big|\cdot\|x_j^0\| \\
	\nono&& +\sum_{s=1}^t\sum_{j=1}^N\big|[P(t-1,s)]_{ij}-\f{1}{N}\big|\cdot\|\e_j^{s-1}+r_{s-1}\xi_j^{s-1}\| \\
	\nono&& \leq\ooo\g^{t-1}\sum_{i=1}^N\|x_j^0\|+\sum_{s=1}^{t-1}\ooo\g^{t-s-1}\sum_{j=1}^N\|\e_j^{s-1}+r_{s-1}\xi_j^{s-1}\| \\
	\nono&&\   \ +\f{1}{N}\sum_{j=1}^N\|\e_j^{t-1}+r_{t-1}\xi_j^{t-1}\|+\|\e_i^{t-1}+r_{t-1}\xi_i^{t-1}\|.
	\eea
	Since $\mbb E[\|\xi_j^{t-1}\|]=\mbb E[\|\sum_{l=1}^N[P^{t-1}]_{jl}\xi_{jl}^{t-1}\|]\leq\sum_{l=1}^N\mbb E[\|\xi_{jl}^{t-1}\|]\leq N\sqrt{\nu}$,
	then $\mbb E [\|\e_j^{t-1}+r_{t-1}\xi_j^{t-1}\|]\leq\mbb E[\|\e_j^{t-1}\|]+r_{t-1}\mbb E[\|\xi_j^{t-1}\|]\leq\f{G_f+G_{\chi}}{\sigma_{\ff}}\aaa_t+N\sqrt{\nu}r_{t-1}$.
	Combine these inequalities, it follows that, for any $i\in\huaV$,
	\bea\label{dis}
	\nono&&\mbb E[\|x_i^t-\bar{x}^t\|]\leq\ooo\g^{t-1}\sum_{j=1}^N\|x_j^0\|\\
	\nono&&\quad\quad\quad\quad\quad\quad\quad+N\sum_{s=1}^{t-1}\ooo\g^{t-s-1}\big(\f{G_f+G_{\chi}}{\sigma_{\ff}}\aaa_s+N\sqrt{\nu}r_s\big)\\
	&&\quad\quad\quad\quad\quad\quad\quad+2\big(\f{G_f+G_{\chi}}{\sigma_{\ff}}\aaa_t+N\sqrt{\nu}r_{t-1}\big).
	\eea
	Sum up both sides of above inequality from $t=1$ to $T$ and $i=1$ to $N$, it follows that
	\bea
	\nono&&\sum_{t=1}^T\sum_{i=1}^N\mbb E[\|x_i^t-\bar{x}^t\|]\leq\f{N\ooo}{1-\g}\sum_{j=1}^N\|x_j^0\|\\
	&&+\big(2N+\f{N^2\ooo}{1-\g}\big)\sum_{t=0}^T\big(\f{G_f+G_{\chi}}{\sigma_{\ff}}\aaa_t+N\sqrt{\nu}r_t\big).\label{sumdis}
	\eea
	Note that the bound on right hand side of (\eqref{dis}) does not depend on the index $i$. For any index $j\in\huaV$, $\mbb E[\|x_j^t-\bar{x}^t\|]$ also satisfies the bound in (\eqref{dis}). Sum up from $t=1$ to $T$ and $i=1$ to $N$ to $\mbb E[\|x_j^t-\bar{x}^t\|]$, use the triangle inequality $\mbb E[\|x_i^t-x_j^t\|]\leq\mbb E[\|x_i^t-\bar{x}^t\|]+\mbb E[\|x_j^t-\bar{x}^t\|]$, and combine with (\eqref{sumdis}), the result in theorem is obtained.
\end{proof}

\begin{lem}
	Let $\{x_i^t\}$, $\{y_i^t\}$ be the sequences in DSCMD-N. Let $\{\aaa_t\}$ be a non-increasing stepsize. Then we have
	\bea
	\nono&&\chi_i(x_i^{t+1})-\chi_i(x^*)+\nn\ww{g}_i^t,y_i^t-x^*\mm\\
	&&\leq\f{1}{\aaa_t}\big[D_{\ff}(x^*,y_i^t)-D_{\ff}(x^*,x_i^{t+1})\big]+\f{\aaa_t}{2\sigma_{\ff}}\|\ww{g}_i^t\|^2.\label{inner}
	\eea
\end{lem}
\begin{proof}
	According to the first-order optimality of the DSCMD-N, there exists $h_i^{t+1}\in\partial\chi_i(x_i^{t+1})$,
	\be
	\nono\nn\aaa_t\ww{g}_i^t+\nb\ff(x_i^{t+1})-\nb\ff(y_i^t)+\aaa_t h_i^{t+1},x-x_i^{t+1}\mm\geq0, \  \forall x\in\huaX.
	\ee
	Set $x=x^*$ in above inequality, and rearrange terms, we have
	\bea
	\nono&&\nn\aaa_t\ww{g}_i^t,x_t^{t+1}-x^*\mm   \\
	\nono&&\leq\nn\nb\ff(y_i^t)-\ff(x_i^{t+1}),x_i^{t+1}-x^*\mm+\aaa_t\nn h_i^{t+1},x^*-x_i^{t+1}\mm  \\
	\nono&&\leq D_{\ff}(x^*,y_i^t)-D_{\ff}(x^*,x_i^{t+1})-D_{\ff}(x^*,y_i^t)\\
	\nono&&\ \ \ +\chi_i(x^*)-\chi_i(x_i^{t+1})\\
	\nono&&\leq D_{\ff}(x^*,y_i^t)-D_{\ff}(x^*,x_i^{t+1})-\f{\sigma_{\ff}}{2}\|x_i^{t+1}-y_i^t\|^2\\
	&&\ \ \ +\chi_i(x^*)-\chi_i(x_i^{t+1}),\label{nei1}
	\eea
	in which the second inequality follows from the three point inequality and the second inequality follows from the definition of $D_{\ff}(\cdot,\cdot)$ and $\sigma_{\ff}$-strong convexity of $\ff$. Also,
	\bea
	\nono&&\nn\aaa_t\ww{g}_i^t,x_t^{t+1}-x^*\mm\geq\nn\aaa_t\ww{g}_i^t,x_t^{t+1}-y_i^t\mm+\nn\aaa_t\ww{g}_i^t,y_i^t-x^*\mm\\
	\nono&&\quad\quad\quad\quad\quad\quad\quad\ \ \geq-\f{\aaa_t^2}{2\sigma_{\ff}}\|\ww{g}_i^t\|^2-\f{\sigma_{\ff}}{2}\|x_i^{t+1}-y_i^t\|^2\\
	&&\quad\quad\quad\quad\quad\quad\quad\ \ \ +\nn\aaa_t\ww{g}_i^t,y_i^t-x^*\mm. \label{nei2}
	\eea
	Combine (\eqref{nei1}) and (\eqref{nei2}), it follows that
	\bea
	\nono&&\nn\aaa_t\w{g}_i^t,y_i^t-x^*\mm\leq D_{\ff}(x^*,y_i^t)-D_{\ff}(x^*,x_i^{t+1})\\
	\nono&& \quad\quad\quad +\f{\aaa_t^2}{2\sigma_{\ff}}\|\ww{g}_i^t\|^2+\aaa_t\big[\chi_i(x^*)-\chi_i(x_i^{t+1})\big].
	\eea
	The proof is concluded after dividing both sides by $\aaa_t$ in above inequality.
\end{proof}

\begin{lem}\label{jianoi}
	Let $\{x_i^t\}$, $\{y_i^t\}$ be the sequences in DSCMD-N, then there holds $\mbb E[\|y_i^t-x_l^t\|]\leq\sum_{j=1}^N\mbb E[\|x_j^t-x_l^t\|]+N\sqrt{\nu}r_t$ for any $i,l\in\huaV$.
\end{lem}
\begin{proof}
	According to the structure of DSCMD-N and the fact that the matrix $P^t$ is doubly stochastic,
	\bea
	\nono&&\|y_i^t-x_l^t\|=\|\sum_{j=1}^N[P^t]_{ij}[x_j^t-x_l^t]+r_t\sum_{j=1}^N[P^t]_{ij}\xi_{ij}^t\| \\
	\nono&&\quad\quad\quad\quad \ \leq\sum_{j=1}^N[P^t]_{ij}\|x_j^t-x_l^t\|+r_t\sum_{j=1}^N\|\xi_{ij}^t\|.
	\eea
	Take expectation over $\huaF_{t-1}$, use Assumption \ref{noise} and the fact that $0\leq[P^t]_{ij}<1$, then take total expectation, the lemma is concluded.
\end{proof}

\begin{lem}\label{Brelem}
	Let $\{x_j^t\}$ be the  sequences in DSCMD-N,  the noise sequence $\{r_t\xi_{ij}^t\}$ is defined as before, we have $\mbb E[D_{\ff}(x^*,x_j^t+r_t\xi_{ij}^t)]\leq \mbb E[D_{\ff}(x^*,x_j^t)]+\sqrt{\f{2}{\sigma_{\ff}}}D_{\ff,\huaX}L_{\ff}\sqrt{\nu}r_t+L_{\ff}\nu r_t^2$.
\end{lem}
\begin{proof}
	According to mean value formula, there exists a $\zeta\in[0,1]$ such that $\ff(x_j^t+r_t\xi_{ij}^t)=\ff(x_j^t)+\nn\nb\ff(x_j^t+\zeta r_t\xi_{ij}^t),r_t\xi_{ij}^t\mm$, then it follows that
	\bea
	\nono&&D_{\ff}(x^*,x_j^t+r_t\xi_{ij}^t)\\
	\nono&&=\ff(x^*)-\ff(x_j^t+r_t\xi_{ij}^t)-\nn\nb\ff(x_j^t+r_t\xi_{ij}^t),x^*-x_j^t-r_t\xi_{ij}^t\mm\\
	\nono&&=\ff(x^*)-\ff(x_j^t)-\nn\nb\ff(x_j^t+\zeta r_t\xi_{ij}^t),r_t\xi_{ij}^t\mm\\
	\nono&&\ \ \ -\nn\nb\ff(x_j^t+r_t\xi_{ij}^t),x^*-x_j^t-r_t\xi_{ij}^t\mm\\
	\nono&&=\ff(x^*)-\ff(x_j^t)-\nn\nb\ff(x_j^t+\zeta r_t\xi_{ij}^t)-\nabla\ff(x_j^t+r_t\xi_{ij}^t),r_t\xi_{ij}^t\mm\\
	\nono&&\ \ \  -\nn\nb\ff(x_j^t),x^*-x_j^t\mm-\nn\nb\ff(x_j^t+r_t\xi_{ij}^t)-\nb\ff(x_j^t),x^*-x_j^t\mm\\
	\nono&&\leq D_{\ff}(x^*,x_j^t)+(1-\zeta)L_{\ff}r_t^2\|\xi_{ij}^t\|^2+L_{\ff}D_{\huaX}r_t\|\xi_{ij}^t\|\\
	\nono&&\leq D_{\ff}(x^*,x_j^t)+D_{\huaX}L_{\ff}\|\xi_{ij}^t\|r_t+L_{\ff}\|\xi_{ij}^t\|^2r_t^2,
	\eea
	in which the first inequality follows from Cauchy inequality and gradient $L_{\ff}$-Lipschitz condition of $\ff$, the second inequality follows from the fact $0\leq1-\zeta\leq1$. Take conditional expectation on $\huaF_{t-1}$ on both sides, use Assumption \ref{noise} and note that $D_{\huaX}\leq\sqrt{\f{2}{\sigma_{\ff}}}D_{\ff,\huaX}$, the result is obtained after taking total expectation.
\end{proof}

Now return to (\eqref{inner}), take conditional expectation over $\huaF_{t-1}$ on both sides of (\eqref{inner}), we have
\bea
\nono &&\nn g_i(y_i^t),y_i^t-x^* \mm+\mbb E[\chi_i(x_i^{t+1})|\huaF_{t-1}]-\chi_i(x^*)\\
\nono&&\leq\f{1}{\aaa_t}\Big[D_{\ff}(x^*,y_i^t)-\mbb E[D_{\ff}(x^*,x_i^{t+1})|\huaF_{t-1}]\Big]\\
\nono&&\ \ \ +\f{\aaa_t}{2\sigma_{\ff}}\mbb E[\|\ww{g}_i^t\|^2|\huaF_{t-1}].
\eea
Take total expectation on both sides of above inequality, we have
\be
\Delta_{i,t}^1+\Delta_{i,t}^2\leq\Delta_{i,t}^3,\label{d3}
\ee
in which we denote
\bea
\nono&&\Delta_{i,t}^1=\mbb E[\nn g_i(y_i^t),y_i^t-x^*\mm],\\
\nono&&\Delta_{i,t}^2=\mbb E[\chi_i(x_i^{t+1})]-\chi_i(x^*),\\
\nono&&\Delta_{i,t}^3=\f{1}{\aaa_t}\big[\mbb E[D_{\ff}(x^*,y_i^t)]-\mbb E[D_{\ff}(x^*,x_i^{t+1})]\big]+\f{\aaa_t}{2\sigma_{\ff}}\mbb E[\|\ww{g}_i^t\|^2].
\eea

Before coming to the main result, we need the following lemma for $\Delta_{i,t}^3$.

\begin{lem}\label{d3up}
	Under Assumptions 1-4, if $\{\aaa_t\}$,  $\{r_t\}$ be non-increasing positive sequences, then the following bound result for $\Delta_{i,t}^3$ holds,
	\bea
	\nono&&\quad\sum_{t=1}^T\sum_{i=1}^N[\Delta_{i,t}^3]\leq\f{ND_{\ff,\huaX}^2}{\aaa_T}+\f{NG_f^2}{2\sigma_{\ff}}\sum_{t=0}^T\aaa_t\\
	&&\ \ \ \ \ \ \ \ +N\sqrt{\f{2}{\sigma_{\ff}}}D_{\ff,\huaX}L_{\ff}\sqrt{\nu}\sum_{t=1}^T\f{r_t}{\aaa_t}+NL_{\ff}\nu\sum_{t=1}^T\f{r_t^2}{\aaa_t}.\label{sumd3}
	\eea
\end{lem}
\begin{proof}
	Since $y_i^t=\sum_{j=1}^N[P^t]_{ij}(x_{j}^t+r_t\xi_{ij}^t)$, separate convexity of $D_{\ff}(\cdot,\cdot)$ implies that
	\bea
	\nono&&\sum_{t=1}^T\sum_{i=1}^N[\Delta_{i,t}^3]\\
	\nono&&\leq\sum_{t=1}^T\f{1}{\aaa_t}\Big[\sum_{i=1}^N\sum_{j=1}^N[P^t]_{ij}\mbb E[D_{\ff}(x^*,x_j^t+r_t\xi_{ij}^t)]-\sum_{i=1}^N\mbb E[D_{\ff}(x^*,x_i^{t+1})]\Big]\\
	\nono&&\ \ \ +\f{NG_f^2}{2\sigma_{\ff}}\sum_{t=0}^T\aaa_t \\
	\nono&&\leq\sum_{t=1}^T\f{1}{\aaa_t}\Big[\sum_{j=1}^N\mbb E[D_{\ff}(x^*,x_j^t)]-\sum_{i=1}^N\mbb E[D_{\ff}(x^*,x_i^{t+1})]\Big]\\
	\nono&&\ \ \ +N\big[2G_{\ff}+\sqrt{\f{2}{\sigma_{\ff}}}D_{\ff,\huaX}L_{\ff}\big]\sqrt{\nu}\sum_{t=1}^T\f{r_t}{\aaa_t}\\
	\nono&&\quad +NL_{\ff}\nu\sum_{t=1}^T\f{r_t^2}{\aaa_t}+\f{NG_f^2}{2\sigma_{\ff}}\sum_{t=0}^T\aaa_t.\\
	\nono&&=\sum_{i=1}^N\Big[\f{1}{\aaa_1}\mbb E[D_{\ff}(x^*,x_i^1)]+\sum_{t=2}^T\mbb E[D_{\ff}(x^*,x_i^t)](\f{1}{\aaa_t}-\f{1}{\aaa_{t-1}})\\
	\nono&&\ \ -\f{1}{\aaa_T}\mbb E[D_{\ff}(x^*,x_i^{T+1})]\Big]+N\sqrt{\f{2}{\sigma_{\ff}}}D_{\ff,\huaX}L_{\ff}\sqrt{\nu}\sum_{t=1}^T\f{r_t}{\aaa_t}\\
	\nono&&\ \ +NL_{\ff}\nu\sum_{t=1}^T\f{r_t^2}{\aaa_t}+\f{NG_f^2}{2\sigma_{\ff}}\sum_{t=0}^T\aaa_t,
	\eea
	in which the second inequality is obtained by double stochasticity of matrix $P^t$ and Lemma \ref{Brelem}, the result is obtained after eliminating same terms in the summation in above equality.
\end{proof}

Now we are ready to give the main result of this section. Denote
\bea
\nono \hat{x}_l^T=\f{1}{T}\sum_{t=1}^Tx_l^t,
\eea
and
\bea
\nono x^*=\arg\min_{x\in\huaX}f(x).
\eea
The following result describes the expected bound for DSCMD-N in terms of stepsizes $\{\aaa_t\}$, noise decaying rates $\{r_t\}$.
\begin{thm}\label{thm1}
	Let the Assumptions 1-4 hold. If $\{\aaa_t\}$,  $\{r_t\}$ are positive non-increasing sequences, then for DSCMD-N method, for any $l\in\huaV$, we have
	\bea
	\nono&&\mbb E[F(\hat{x}_l^T)]-F(x^*)\leq \f{C_1}{T}+\f{C_2}{T\aaa_T}+\f{C_3}{T}\sum_{t=0}^T\aaa_t+\f{C_4}{T}\sum_{t=0}^Tr_t\\
	&&\ \ \ \ \ \ \ \ \ \ \ \ \ \ \ \ \ \ \ \ \ \ \ \ \ \ +\f{C_5}{T}\sum_{t=0}^T\f{r_t}{\aaa_t}+\f{C_6}{T}\sum_{t=0}^T\f{r_t^2}{\aaa_t},
	\eea
	in which
	\bea
	\nono&&C_1=\f{2N\ooo}{1-\g}[(N+1)G_f+NG_{\chi}]\cdot\|x_j^0\|, \ C_2=ND_{\ff,\huaX}^2,\\
	\nono&&C_3=\bigg(\big(4N+\f{2N^2\ooo}{1-\g}\big)[(N+1)G_f+NG_{\chi}]\\
	\nono&& \ \ \ \ \ \  +NG_{\chi}\bigg)\cdot\f{G_f+G_{\chi}}{\sigma_{\ff}}+\f{NG_f^2}{2\sigma_{\ff}},\\
	\nono&&C_4=\big(4N+\f{2N^2\ooo}{1-\g}\big)[(N+1)G_f+NG_{\chi}]N\sqrt{\nu}\\
	\nono&& \ \ \ \ \ \ +(G_f+G_{\chi})N^2\sqrt{\nu},\\
	\nono&&C_5=\sqrt{\f{2}{\sigma_{\ff}}}D_{\ff,\huaX}L_{\ff}N\sqrt{\nu}, \ C_6=NL_{\ff}\nu,
	\eea
	and $\ooo=(1-\f{\ttt}{4N^2})^{-2}$, $\gamma=(1-\f{\ttt}{4N^2})^{\f{1}{B}}$.
\end{thm}

\begin{proof}
	We prove the result by estimating the terms in (\eqref{d3}). For any index $l\in\huaV$,
	\bea
	\nono&&\nn g_i(y_i^t),y_i^t-x^*\mm\\
	\nono&&\geq f_i(y_i^t)-f_i(x^*)\\
	\nono&&=f_i(y_i^t)-f_i(x_l^t)+f_i(x_l^t)-f_i(x^*)\\
	\nono&&\geq-G_f\|y_i^t-x_l^t\|+[f_i(x_l^t)-f_i(x^*)]\\
	\nono&&\geq -G_f\|y_i^t-x_i^t\|-G_f\|x_i^t-x_l^t\|+[f_i(x_l^t)-f_i(x^*)].
	\eea
	In which the second inequality follows from $g_i(x)=\mbb E[\ww{g}_i(x)|\huaF_{t-1}]\leq\sqrt{\mbb E[\|\ww{g}_i(x)\|^2|\huaF_{t-1}]}\leq G_f$.                               After taking expectation and using Lemma \ref{jianoi}, it follows that $\Delta_{i,t}^1\geq-G_f\sum_{j=1}^N\mbb E[\|x_j^t-x_i^t\|]-N\sqrt{\nu}G_fr_t-G_f\mbb E[\|x_i^t-x_l^t\|]+\mbb E[f_i(x_l^t)-f_i(x^*)]$. Denote $f=\sum_{i=1}^Nf_i$, $\chi=\sum_{i=1}^N\chi_i$ and denote the bound on the right hand side in Lemma \ref{disagreement} by $\huaB_T$, sum up both sides and use Lemma \ref{disagreement}, it follows that
	\bea
	\nono &&\sum_{t=1}^T\sum_{i=1}^N\big[\Delta_{i,t}^1\big]\geq-(N+1)G_f\huaB_T-N^2\sqrt{\nu}G_f\sum_{t=1}^Tr_t\\
	&& \ \ \ \ \ \  \ \ \ \ \ \ \ \ \ \ \ \ \ \ +\sum_{t=1}^T\mbb E[f(x_l^t)-f(x^*)].\label{sumd1}
	\eea
	On the other hand, for any index $l\in\huaV$,
	\bea
	\nono&&\chi_i(x_i^{t+1})-\chi_i(x^*)\\
	\nono&&=[\chi_i(x_i^{t+1})-\chi_i(y_i^t)]+[\chi_i(y_i^t)-\chi_i(x_l^t)]+[\chi_i(x_l^t)-\chi_i(x^*)]\\
	\nono&&\geq -G_{\chi}\|x_i^{t+1}-y_i^t\|-G_{\chi}\|y_i^t-x_l^t\|+[\chi_i(x_l^t)-\chi_i(x^*)].
	\eea
	After taking expectation on both sides, using Lemma \ref{error} and Lemma \ref{jianoi}, we have $\Delta_{i,t}^2\geq-G_{\chi}\f{G_f+G_{\chi}}{\sigma_{\ff}}\aaa_t-G_{\chi}\sum_{j=1}^N\mbb E[\|x_j^t-x_l^t\|]-G_{\chi}N\sqrt{\nu}r_t+\mbb E[\chi_i(x_l^t)-\chi_i(x^*)]$. Sum up from $i=1$ to $N$ and $t=1$ to $T$ on both sides, we obtain
	\bea
	\nono&&\sum_{t=1}^T\sum_{i=1}^N\big[\Delta_{i,t}^2\big]\geq-NG_{\chi}\f{G_f+G_{\chi}}{\sigma_{\ff}}\sum_{t=1}^T\aaa_t-NG_{\chi}\huaB_T\\
	&&\quad\quad -N^2G_{\chi}\sqrt{\nu}\sum_{t=1}^Tr_t+\sum_{t=1}^T\big[\chi(x_l^t)-\chi(x^*)\big].\label{sumd2}
	\eea
	Sum up both sides of (\eqref{d3}) from $i=1$ to $N$ and $t=1$ to $T$, combine it with (\eqref{sumd1}), (\eqref{sumd2}), Lemma \ref{d3up}. The desired result is obtained after substituting $\huaB_T$, using Lemma \ref{disagreement}, dividing both sides by $T$ and using the convexity of $F_i$, $i=1,2,...,N$.
\end{proof}

Under a boundedness assumption of stochastic gradient and network noise, the following high probability bound holds for DSCMD-N.
\begin{thm}\label{DSMDpro}
	Under the assumptions of Theorem \ref{thm1}, if we assume in addition that $\|\ww{g}_i^t\|\leq G_f$ and $\|\xi_{ij}^t\|^2\leq \nu$, then for DSCMD-N, for any $l\in\huaV$, we have, for any $\delta\in(0,1)$, with probability of at least $1-\delta$, there holds
	\bea
	\nono &&F(\hat{x}_l^T)-F(x^*)\leq \f{C_1}{T}+\f{C_2}{T\aaa_T}+\f{C_3}{T}\sum_{t=0}^T\aaa_t+\f{C_4}{T}\sum_{t=0}^Tr_t\\
	\nono&&+\f{C_5}{T}\sum_{t=0}^T\f{r_t}{\aaa_t}+\f{C_6}{T}\sum_{t=0}^T\f{r_t^2}{\aaa_t}+2\sqrt{2}G_fD_{\huaX}N\f{\sqrt{\log(1/\delta)}}{\sqrt{T}},
	\eea
	in which $C_1\sim C_6$ are defined as in Theorem \ref{thm1}.
\end{thm}
\begin{proof}
	For saving space, we just show the difference between the proof for this result and the above expected bound result. Come back to (\eqref{inner}), if we denote $\bar{\Delta}_{i,t}^1=\nn g_i(x_i^t),y_i^t-x^*\mm$, $\bar{\Delta}_{i,t}^2=\chi_i(x_i^{t+1})-\chi_i(x^*)$, $\bar{\Delta}_{i,t}^3=\f{1}{\aaa_t}\big[D_{\ff}(x^*,y_i^t)-D_{\ff}(x^*,x_i^{t+1})\big]$, $X_{i,t}=\nn g_i(y_i^t)-\ww{g}_i^t,y_i^t-x^*\mm$, then (\eqref{inner}) can be written in the form of $\bar{\Delta}_{i,t}^1+\bar{\Delta}_{i,t}^2\leq\bar{\Delta}_{i,t}^3+X_{i,t}$. $\bar{\Delta}_{i,t}^1$, $\bar{\Delta}_{i,t}^2$, $\bar{\Delta}_{i,t}^3$ corresponds to $\Delta_{i,t}^1$, $\Delta_{i,t}^2$, $\Delta_{i,t}^3$ in (\eqref{d3}) only up to a procedure of taking expectation. If we denote $X_t=\sum_{i=1}^NX_{i,t}$ and sum up both sides from $t=1$ to $T$ and $i=1$ to $N$, it follows that
	\be
	\sum_{t=1}^T\sum_{i=1}^N[\bar{\Delta}_{i,t}^1]+\sum_{t=1}^T\sum_{i=1}^N[\bar{\Delta}_{i,t}^2]\leq\sum_{t=1}^T\sum_{i=1}^N[\bar{\Delta}_{i,t}^3]+\sum_{t=1}^TX_{t}.\label{sumx}
	\ee
	Note that $\mbb E[X_t|\huaF_{t-1}]=0$, the bound condition $\|\ww{g}_i^t\|\leq G_f$ and Cauchy inequality implies $|X_t|\leq2NG_fD_{\huaX}$, then $\{X_t\}$ is a bounded martingale difference sequence. Use Azuma-Hoeffding inequality (Lemma \ref{azuma}) to $\{X_t\}$, we have for any $\eee>0$
	\be
	\mathrm{Prob}\big(\sum_{t=1}^TX_t\geq\eee\big)\leq\exp\big(-\f{\eee^2}{2T(2G_fD_{\huaX}N)^2}\big).
	\ee
	Setting the above probability upper bound to $\delta$, we have, with probability at least $1-\delta$,
	\be
	\sum_{t=1}^TX_t\leq2\sqrt{2}G_fD_{\huaX}N\sqrt{T}\sqrt{\log\f{1}{\delta}}.\label{neipro}
	\ee
	On the other hand, it is easy to see that, with bound assumptions $\|\ww{g}_i^t\|\leq G_f$ and $\|\xi_{ij}^t\|^2\leq\nu$ in hand, the estimate result of Lemma \ref{disagreement} and Lemma \ref{Brelem} holds without taking expectation. Therefore we know (\eqref{sumd3}), (\eqref{sumd1}), (\eqref{sumd2}) hold with $\Delta_{i,t}^1$, $\Delta_{i,t}^2$, $\Delta_{i,t}^3$ replaced by $\bar{\Delta}_{i,t}^1$, $\bar{\Delta}_{i,t}^2$, $\bar{\Delta}_{i,t}^3$. Combining these three estimates with (\eqref{neipro}) and (\eqref{sumx}), dividing both sides by $T$ and using the convexity of $F_i$, $i=1,2,...,N$, we obtain the desired result.
\end{proof}

\section{Convergence rates of DSCMD-N}
In this section, we provide a general framework for convergence rate analysis by selecting different stepsizes under different effects of noise decaying rates $\{r_t\}$. We also show that, in some situations of $\{r_t\}$, by selecting some stepsizes of $\{\aaa_t\}$, the best achievable rate of $O(\f{1}{\sqrt{T}})$ for centralized subgradient method for nonsmooth convex optimization, can be obtained for DSCMD-N. We present the results on expected rate and high probability rate in the following section.

The following proposition provides a general expected bound for expected error $\mbb E[F(\hat{x}_l^T)]-F(x^*)$ in terms of the total iteration step $T$ with a general stepsize consideration in form of $\aaa_t=\f{1}{(t+1)^{\kk_1}}$ and  noise decaying rate in form of $r_t=\f{1}{(t+1)^{\kk_2}}$.
\begin{pro}\label{MDgen}
	Under conditions of Theorem \ref{thm1}, if the  sequences $\{\aaa_t\}$, $\{r_t\}$  in the DSCMD-N method are $\aaa_t=\f{1}{(t+1)^{\kk_1}}$ and $r_t=\f{1}{(t+1)^{\kk_2}}$, $t=1,2,...,T$. Suppose that $0<\kk_1<\kk_2\leq1$ and $2\kk_2-\kk_1\neq1$, then for any $l\in \huaV$, we have
	\bea
	\nono&&\mbb E[F(\hat{x}_l^T)]-F(x^*)\\
	\nono&&\leq\bigg(C_1+\bigg|\f{1-2(2\kk_2-\kk_1)}{1-(2\kk_2-\kk_1)}\bigg|C_6\bigg)\f{1}{T}+2^{\kk_1}C_2\f{1}{T^{1-\kk_1}}\\
	\nono&&\ \ \ \ +\f{2^{1-\kk_1}C_3}{1-\kk_1}\f{1}{T^{\kk_1}}+\f{2^{1-\kk_2}C_4}{1-\kk_2}\f{1}{T^{\kk_2}}\\
	\nono&&\ \ \ +\f{2^{1-(\kk_2-\kk_1)}C_5}{1-(\kk_2-\kk_1)}\f{1}{T^{\kk_2-\kk_1}}+\f{C_6}{|1-(2\kk_2-\kk_1)|}\f{1}{T^{2\kk_2-\kk_1}},\\
	\nono&& \ \ if \ \kk_1\in(0,1), \ \kk_2\in(0,1);\ and \\
	\nono&& \leq \big( C_1+\f{2-\kk_1}{1-\kk_1}C_6\big)\f{1}{T}+\big(2^{\kk_1}C_2+\f{2^{\kk_1}C_5}{\kk_1}\big)\f{1}{T^{1-\kk_1}}\\
	\nono&&\ \ +\f{2^{1-\kk_1}C_3}{1-\kk_1}\f{1}{T^{\kk_1}}+4C_4\f{\ln T}{T},\ \ \ \ \  if \ \ \kk_1\in(0,1), \ \kk_2=1,
	\eea
	in which $C_1\sim C_6$ are defined as in Theorem \ref{thm1}.
\end{pro}
\begin{proof}
	See Appendix.
\end{proof}

The following result provides a class of novel convergence rates for a general class of noise decaying rate.
\begin{cor}\label{tuitt}
	Under conditions of Theorem \ref{thm1}, suppose the  sequences $\{\aaa_t\}$, $\{r_t\}$  in the DSCMD-N method are $\aaa_t=\f{1}{\sqrt{t+1}}$ and $r_t=\f{1}{(t+1)^{\kk}}$, with $\kk\in (\f{1}{2},1)$ and $\kk\neq\f{3}{4}$.  If the constant $C_{\kk}$ is taken as
	\bea
	\nono&&C_{\kk}=\max\Big\{C_1+\Big|\f{1-2(2\kk-\f{1}{2})}{1-(2\kk-\f{1}{2})}\Big|C_6,\sqrt{2}C_2+2\sqrt{2}C_3,\\
	\nono&&\ \ \ \ \ \ \ \ \ \ \ \ \ \ \ \ \f{2^{1-\kk}C_4}{1-\kk},\f{2^{\f{3}{2}-\kk}C_5}{\f{3}{2}-\kk},\f{C_6}{|\f{3}{2}-2\kk|}\Big\}
	\eea
	then for any $l\in\huaV$, $T\geq3$,  the DSCMD-N method achieves an expected rate of $O(\f{1}{T^{\kk-\f{1}{2}}})$ in following sense:
	\be
	\mbb E[F(\hat{x}_l^T)]-F(x^*)\leq 5C_{\kk}/T^{\kk-\f{1}{2}}.
	\ee
	in which $C_1\sim C_6$ are defined as in Theorem \ref{thm1}.
\end{cor}
\begin{proof}
	By using Proposition \ref{MDgen} to $\kk_1=\f{1}{2}$ and $\kk_2=\kk$, it follows that
	\bea
	\nono&&\mbb E[F(\hat{x}_l^T)]-F(x^*)\leq\\
	\nono&&\Big(C_1+\Big|\f{1-2(2\kk-\f{1}{2})}{1-(2\kk-\f{1}{2})}\Big|C_6\Big)\f{1}{T}+\Big(\sqrt{2}C_2+2\sqrt{2}C_3\Big)\f{1}{\sqrt{T}}\\
	\nono&& +\f{2^{1-\kk}C_4}{1-\kk}\f{1}{T^{\kk}}+\f{2^{\f{3}{2}-\kk}C_5}{\f{3}{2}-\kk}\f{1}{T^{\kk-\f{1}{2}}}+\f{C_6}{|\f{3}{2}-2\kk|}\f{1}{T^{2\kk-\f{1}{2}}}.
	\eea
	Note that, when $\kk\in (\f{1}{2},1)$, there holds
	\be
	\nono \f{1}{T}<\f{1}{T^{\kk}}<\f{1}{\sqrt{T}}<\f{1}{T^{\kk-\f{1}{2}}},
	\ee
	and
	\be
	\nono \f{1}{T^{2\kk-\f{1}{2}}}<\f{1}{T^{\kk-\f{1}{2}}}.
	\ee
	Therefore, after taking maximum coefficient $C_{\kk}$ as above, the desired result holds.
\end{proof}

The following corollary shows a selection of $\aaa_t$ such that the DSCMD-N achieve the optimal rate in expectation under the case when the network has a communication noise decaying rate $r_t=\f{1}{t+1}$.
\begin{cor}\label{tuitui}
	Under conditions of Theorem \ref{thm1}, suppose the  sequences $\{\aaa_t\}$, $\{r_t\}$  in the DSCMD-N method are $\aaa_t=\f{1}{\sqrt{t+1}}$ and $r_t=\f{1}{t+1}$, $t=1,2,...,T$.  If we take $C=\max\{C_1+3C_6,4C_4,\sqrt{2}C_2+2\sqrt{2}C_3+2\sqrt{2}C_5\}$, then for any $l\in\huaV$, $T\geq3$,  the DSCMD-N method achieves an expected rate of $O(\f{1}{\sqrt{T}})$ as follow:
	\be
	\mbb E[F(\hat{x}_l^T)]-F(x^*)\leq 3C/\sqrt{T}.
	\ee
	in which $C_1\sim C_6$ are defined as in Theorem \ref{thm1}.
\end{cor}

\begin{proof}
	By using Proposition \ref{MDgen} to the case when $\kk_1=1/2$ and $\kk_2=1$, we have
	\bea
	\nono&&\mbb E[F(\hat{x}_l^T)]-F(x^*)\leq \big(C_1+3C_6\big)\f{1}{T}+4C_4\f{\ln T}{T}\\
	\nono&&\ \ \ \ \ \ +\big(\sqrt{2}C_2+2\sqrt{2}C_3+2\sqrt{2}C_5\big)\f{1}{\sqrt{T}}.
	\eea
	After taking the maximum of the coefficients $C=\max\{C_1+3C_6,4C_4,\sqrt{2}C_2+2\sqrt{2}C_3+2\sqrt{2}C_5\}$ and noting that $\f{1}{T}\leq\f{\ln T}{T}\leq\f{1}{\sqrt{T}}$ when $T\geq3$, the result is obtained.
\end{proof}
\begin{rmk}
	In fact, for a general order pair $(\kk_1,\kk_2)$ of  $\aaa_t=O(\f{1}{(t+1)^{\kk_1}})$, $\kk_1\in(0,1)$ and $r_t=O(\f{1}{(t+1)^{\kk_2}})$, $\kk_2\in(0,1]$. $(\kk_1,\kk_2)=(1/2,1)$ is the unique pair of $(\kk_1,\kk_2)$ such that the convergence rate becomes $O(1/\sqrt{T})$. For other case, they are worse than this rate. Since for $\kk_1\in(0,1)$ and $\kk_2\in(0,1)$, by using similar idea with Corollary \ref{tuitui}, we have a rate of $O(\f{1}{T^{\kk}})$ with $\kk=\min\{1-\kk_1,\kk_1,\kk_2,\kk_2-\kk_1,2\kk_2-\kk_1\}=\min\{1-\kk_1,\kk_1,\kk_2-\kk_1\}$. If $0<\kk_1<1/2$, then $\kk=\min\{1-\kk_1,\kk_1,\kk_2-\kk_1\}\leq\kk_1<1/2$ which presents a worse rate. If $\kk_1=1/2$, then $\kk=\min\{1/2,\kk_2-1/2\}=\kk_2-1/2<1/2$, for $\kk_1<\kk_2<1$, which is also a worse rate than $O(1/\sqrt{T})$. Hence, the rate $O(1/\sqrt{T})$ can be obtained only when $(\kk_1,\kk_2)=(1/2,1)$.
\end{rmk}

Next, we consider the high probability convergence rate for DSCMD-N by presenting following results.
\begin{pro}\label{thm7}
	Under conditions of Theorem \ref{DSMDpro}, let the sequences $\{\aaa_t\}$, $\{r_t\}$ in the DSCMD-N method be $\aaa_t=\f{1}{\sqrt{t+1}}$ and $r_t=\f{1}{t+1}$, $t=1,2,...,T$. Then for any $l\in\huaV$, $T\geq3$, we have, for any $\delta\in (0,1)$, with probability of at least $1-\delta$,
	\bea
	\nono&& F(\hat{x}_l^T)-F(x^*)\leq \big(C_1+3C_6\big)\f{1}{T}+4C_4\f{\ln T}{T}\\
	\nono&&+\bigg[\sqrt{2}C_2+2\sqrt{2}C_3+2\sqrt{2}C_5+2\sqrt{2}G_fD_{\huaX}N\sqrt{\log(1/\delta)}\bigg]\f{1}{\sqrt{T}},
	\eea
	in which $C_1\sim C_6$ are defined as in Theorem \ref{thm1}.
\end{pro}
\begin{proof}
	The proof has the similar procedure with Corollary \ref{tuitui} by using the general bounds for terms of $\aaa_t$ and $r_t$. The result is obtained by combining an additional term of $2\sqrt{2}G_fD_{\huaX}N\f{\sqrt{\log(1/\delta)}}{\sqrt{T}}$ (this term appears since we consider high probability bound this time).
\end{proof}

The high probability optimal rate of $O(1/\sqrt{T})$ for DSCMD-N is obtained in the following corollary.

\begin{cor}
	Under conditions of Proposition \ref{thm7}, for any $\delta\in(0,1)$, set $C_{\delta}=\max\{C_1+3C_6,4C_4,\sqrt{2}C_2+2\sqrt{2}C_3+2\sqrt{2}C_5+2\sqrt{2}G_fD_{\huaX}N\sqrt{\log(1/\delta)}\}$. Then for any $l\in\huaV$, $T\geq3$, we have, for any $\delta\in (0,1)$, with probability of at least $1-\delta$, the DSCMD-N method achieves the following rate
	\be
	\nono F(\hat{x}_l^T)-F(x^*)\leq 3C_{\delta}/\sqrt{T}.
	\ee
\end{cor}
\begin{proof}
	The result follows directly from Proposition \ref{thm7}.
\end{proof}

\begin{rmk}
	Now we make a comparison between the results on DSCMD-N in this work and  some main existing works in this literature (\cite{n2}, \cite{n3}). \cite{n3} is a seminal work on distributed optimization over noisy network. Both of the works \cite{n2}, \cite{n3} consider standard distributed Euclidean projection-based algorithms to minimize the objective function $\sum_{i=1}^Nf_i(x)$ associated with local functions $f_i$, $i\in\huaV$. Their approaches rely on a standard Robbins-Monro  stepsize summability condition $\sum_{t=0}^{\infty} \aaa_{t}=\infty$ and $\sum_{t=0}^{\infty}\aaa_t^2<\infty$ to ensure the almost sure convergence of $\{x_i^t\}$ to the solution set $\huaX^*$.  In this work, DSCMD-N method is introduced in a more general setting (composite optimization) when regularization terms are considered. Hence we are able to handle the optimization problem from different angles by selecting different types of regularizers. Also, the Bregman divergence is utilized instead of the Euclidean projection in \cite{n2}, \cite{n3}, therefore, the proposed algorithm can better reflect the geometric feature of the underlying decision space when  selecting different types of mirror map (distance-generating function) $\Phi$.
\end{rmk}
\begin{rmk}
	Here, we mention a special case: when we consider regularizer $\chi_i=0$ and mirror map $\Phi=\f{1}{2}\|\cdot\|^2$, then the algorithm degenerates to \cite{n2} if a zeroth-order gradient oracle is used. Moreover, we relax the aforementioned stepsize assumptions (hence the stepsize $\aaa_t=\f{1}{(t+1)^{\kk}}$ with $\kk\in(1,1/2]$ can be used, this stepsize can not be considered and used  in \cite{n2}, \cite{n3}) and derive the explicit convergence rate in expectation.  On the way to the convergence in expectation, we also relax an assumption of noise $\{\xi_{ij}\}$ in contrast to \cite{n2}. In fact, we do not require the martingale difference condition $\mbb E[\xi_{ij}^t]=0$ to get expectation convergence results.  As an important counterpart of convergence in expectation, high probability bound and rate are also obtained via Azuma-Hoeffding inequality, which enriches the convergence class of distributed optimization methods in this literature.  These convergence rates and bounds are new in noisy network optimization setting.
\end{rmk}

\begin{rmk}
	In contrast to existing works on noisy network optimization, the paper also considers composite terms that serves as regularization terms for the composite optimization problem (local regularizer $\chi_i$, $i=1,2,...,N$ in Problem (\eqref{problem1}). The utilization of the regularization terms makes the method more flexible to present some structure types of the solution of optimization problem. Meanwhile, the structure of Problem (\eqref{problem1}) and DSCMD-N method allow the regularization term $\chi_i$ associated with agent $i$ to be independent of each other. There are several choices of $\chi_i$, $\eta$ that are often considered to promote different structure types of solutions of optimization problem. For example, the indicator function of $\huaX$, $I_{\huaX}(x)$; The $l^p$-norm squared function $\f{1}{2}\|x\|_p^2$, $p\in(1,2]$; Sparsity inducing regularizer $\lambda\|x\|_1$, $\lambda>0$; $l^{\infty}$-norm $\lambda\|x\|_{\infty}$, $\lambda>0$; entropy function $\sum_{i=1}^n[x]_i\log [x]_i$; mixed regularizer $\f{\lambda_1}{2}\|x\|_2^2+\lambda_2\|x\|_1$, $\lambda_1,\lambda_2>0$.
\end{rmk}

Till now, we observe that all the approximating sequences of convergence results in this paper are in weighted average form $\hat{x}_i=\f{1}{T}\sum_{i=1}^Tx_i^t$, $i=1,2,...,N$. The expectation convergence and high probability convergence result are derived. A question rises that, can we present some almost sure convergence results for local sequence $\{\hat{x}_i^t\}$ or $\{x_i^t\}$ in distributed composite optimization setting? To this end, we provide following almost sure convergence results for DSCMD-N.  In the following, we use $\mathrm{dist}(x,M)$ to denote the distance from a point $x$ to the closed set $M$. Namely, $dist(x,M)=\inf\{\|x-m\|:m\in M\}$.

\begin{cor}\label{aec}
	Under conditions of Theorem \ref{thm1}, in the DSCMD-N method, suppose the stepsize sequences $\{\aaa_t\}$, noise decreasing rate $\{r_t\}$  are $\aaa_t=\f{1}{\sqrt{t+1}}$ and $r_t=\f{1}{t+1}$, $t=1,2,...,T$.  If we take $C=\max\{C_1+3C_6,4C_4,\sqrt{2}C_2+2\sqrt{2}C_3+2\sqrt{2}C_5\}$, in which $C_1\sim C_6$ is as in Theorem \ref{thm1}. Then for any $l\in\huaV$, $T\geq3$, for sequence $\{x_i^t\}$ and their average version $\{\hat{x}_i^t\}$ generated from DSCMD-N method, we have almost surely,
	\be
	\nono\underline{\lim}_{T\rightarrow\infty}F(\hat{x}_i^T)=F(x^*), \ \underline{\lim}_{T\rightarrow\infty}\mathrm{dist}(\hat{x}_i^T,\huaX^*)=0, \ \forall i\in\huaV;
	\ee
	and
	\be
	\nono\lim_{T\rightarrow\infty}\min_{1\leq t\leq T}F(x_i^t)=F(x^*), \ \lim_{T\rightarrow\infty}\Big(\min_{1\leq t\leq T}\mathrm{dist}(x_i^t,\huaX^*)\Big)=0, \ \forall i\in\huaV.
	\ee
\end{cor}
\begin{proof}
	See Appendix.
\end{proof}

\section{Conclusion}
This paper has studied a class of noisy network optimization problems. One distributed stochastic composite optimization problems over noisy network are considered. Based on Bregman non-Euclidean projection scheme, a new method DSCMD-N is presented to solve them respectively. Convergence of the methods are systematically studied. New convergence rates are obtained in several different situations under different detailed discussions on stepsize $\{\aaa_t\}$ and communication noise decreasing rate $\{r_t\}$. These new convergence results include expectation convergence, high probability convergence and almost sure convergence. These results enrich the exploration in noisy network optimization. The rates for expectation convergence and high probability convergence are first derived in the literature. Since we have considered randomness on both network links and gradients, the potential value of the methods are obvious in stochastic circumstances. The experiments verify the theoretical results in this paper.

\section{Appendix}
\subsection{Proof of Proposition \ref{MDgen}}
\begin{proof}
	Note that, under the conditions that $0<\kk_1<\kk_2<1$ and $2\kk_2-\kk_1\neq1$, $\f{1}{T\aaa_T}=\f{(T+1)^{\kk_1}}{T}\leq\f{(2T)^{\kk_1}}{T}=\f{2^{\kk_1}}{T^{1-\kk_1}}$, $T\geq1$. $\f{1}{T}\sum_{t=0}^T\aaa_t\leq\f{1}{T}\sum_{t=0}^T\f{1}{T}(1+\sum_{t=1}^T\aaa_t)\leq\f{1}{T}(1+\int_0^T\f{1}{(t+1)^{\kk_1}})\leq\f{1}{T}\f{(T+1)^{1-\kk_1}}{1-\kk_1}\leq\f{2^{1-\kk_1}}{1-\kk_1}\f{1}{T^{\kk_1}}$, similar reason implies $\f{1}{T}\sum_{t=0}^Tr_t\leq\f{2^{1-\kk_2}}{1-\kk_2}\f{1}{T^{\kk_2}}$ and $\f{1}{T}\sum_{t=0}^T\f{r_t}{\aaa_t}\leq\f{2^{1-(\kk_2-\kk_1)}}{1-(\kk_2-\kk_1)}\f{1}{T^{\kk_2-\kk_1}}$. Also, note that $\f{1}{T}\sum_{t=0}^T\f{r_t^2}{\aaa_t}=\f{1}{T}(1+\sum_{t=1}^T\f{1}{(t+1)^{2\kk_2-\kk_1}})\leq\f{1}{T}(2+\sum_{t=2}^T\f{1}{t^{2\kk_2-\kk_1}})\leq\f{1}{T}(2+\int_1^T\f{dt}{t^{2\kk_2-\kk_1}})=\f{1}{T}(2+\f{T^{1-(2\kk_2-\kk_1)}}{1-(2\kk_2-\kk_1)}-\f{1}{1-(2\kk_2-\kk_1)})\leq\big|\f{1-2(2\kk_2-\kk_1)}{1-(2\kk_2-\kk_1)}\big|\f{1}{T}+\f{1}{|1-(2\kk_2-\kk_1)|}\f{1}{T^{2\kk_2-\kk_1}}$. Substituting these bounds into Theorem \ref{thm1}, we obtain the first argument. For $\kk_1\in(0,1)$ and $\kk_2=1$, $\f{1}{T}\sum_{t=0}^Tr_t=\f{1}{T}\sum_{t=0}^T\f{1}{t+1}\leq1+\int_0^T\f{dt}{t+1}=1+\ln (T+1)\leq1+\ln 2T\leq2\ln 2T\leq2(\ln2+\ln T)\leq4\ln T$, $T\geq3$. $\f{1}{T}\sum_{t=0}^T\f{r_t}{\aaa_t}=\f{1}{T}\sum_{t=0}^T\f{1}{(t+1)^{1-\kk_1}}\leq\f{2^{\kk_1}}{\kk_1}\f{1}{T^{1-\kk_1}}$, $\f{1}{T}\sum_{t=0}^T\f{r_t^2}{\aaa_t}=\f{1}{T}\sum_{t=0}^T\f{1}{(t+1)^{2-\kk_1}}\leq1+\int_0^{\infty}\f{dt}{(t+1)^{2-\kk_1}}=\f{2-\kk_1}{1-\kk_1}\f{1}{T}$, the second argument is obtained after substituting these estimates into Theorem \ref{thm1} again.
\end{proof}

\subsection{Proof of Corollary \ref{aec}}
\begin{proof}
	According to Corollary \ref{tuitui}, we have $\mbb E[F(\hat{x}_i^T)]-F(x^*)\leq 3C/\sqrt{T}$, $i\in\huaV$. This implies $\lim_{T\rightarrow\infty}\Big(\mbb E[F(\hat{x}_i^T)]-F(x^*)\Big)=0$. since $F(\hat{x}_i^T)-F(x^*)$ is nonnegative for all $T\geq0$, by applying Fatou lemma we arrive at
	$\underline{\lim}_{T\rightarrow\infty}F(\hat{x}_i^T)=F(x^*), \ a.s.$
	On the other hand, we have already assumed that $\huaX$ is bounded and $F$ is continuous (since $F_i$ is continuous for all $i\in\huaV$). Weierstrass Theorem implies that the accumulation point set of $\{\hat{x}_i^T\}$ exists. The above inequality and the continuity of $F$ implies at least one of the accumulation points minimizes the summation of Problem (\eqref{problem1}), which means
	\be
	\underline{\lim}_{T\rightarrow\infty}\mathrm{dist}(\hat{x}_i^T,\huaX^*)=0, \  i\in\huaV.\label{as1}
	\ee
	Due to the convexity of $F$ and the fact that $\hat{x}_i^T$ is a convex combination of $x_i^1,x_i^2,...,x_i^T$. We have $\min_{1\leq t\leq T}F(x_i^t)\leq F(\hat{x}_i^T)$, $i\in\huaV$. Then it follows that $0\leq\min_{1\leq t\leq T}F(x_i^t)-F(x^*)\leq F(\hat{x}_i^T)-F(x^*)$, $i\in\huaV$. After taking expectation on both sides, we have
	\be
	0\leq\mbb E\Big[\min_{1\leq t\leq T}F(x_i^t)-F(x^*)\Big]\leq \mbb E\Big[F(\hat{x}_i^T)-F(x^*)\Big], i\in\huaV.
	\ee
	Take limit on both sides of above inequality and use Squeeze theorem, we have $\lim_{T\rightarrow\infty}\mbb E\Big[\min_{1\leq t\leq T}F(x_i^t)-F(x^*)\Big]$. Use Fatou lemma again, we have $\underline{\lim}_{T\rightarrow\infty}\Big[\min_{1\leq t\leq T}F(x_i^t)-F(x^*)\Big]=0$. Since $\min_{1\leq t\leq T}F(x_i^t)$ is a lower bounded non-increasing sequence in $T$, hence $\lim_{T\rightarrow\infty}\min_{1\leq t\leq T}F(x_i^t)$ exists and \\
	$\lim_{T\rightarrow\infty}\min_{1\leq t\leq T}F(x_i^t)=\underline{\lim}_{T\rightarrow\infty}\min_{1\leq t\leq T}F(x_i^t)=F(x^*)$, $i\in\huaV$. Then, using similar argument of getting (\eqref{as1}), we arrive at $\lim_{T\rightarrow\infty}\Big(\min_{1\leq t\leq T}\mathrm{dist}(x_i^t,\huaX^*)\Big)=0$, $i\in\huaV$.
\end{proof}


\begin{thebibliography}{}
	\bibitem{s1}
	A. Nedic and A. Ozdaglar,  ``Distributed subgradient methods for multi-agent optimization,'' \emph{IEEE Trans. on Automat. Control}, vol. 54,no. 1, pp. 48-61, 2009.
	
	\bibitem{2016}
	A. Nedic and A. Olshevsky,  ``Stochastic gradient-push for strongly convex functions on time-varying directed graphs,'' \emph{IEEE Trans. on Automat. Control}, vol 61, no. 12, pp. 3936-3947, 2016.
	
	\bibitem{md}
	A. Nedic and S. Lee,  ``On stochastic subgradient mirror-descent algorithm with weighted averaging,'' \emph{SIAM J. Optim.}, vol. 24,no. 1, pp.  84-107, 2014.
	
	
	\bibitem{adver}
	A. Khanafer, B. Touri, T. Basar.  ``Consensus in the presence of an adversary,'' \emph{Proceedings of the 3rd IFAC Workshop on
		Distributed Estimation and Control in Networked Systems}, vol. 45, no. 26, pp. 276-281, 2012.
	
	\bibitem{adv1}
	A. Akhavan, M. Pontil, A. B. Tsybakov. ``Distributed Zero-Order Optimization under Adversarial Noise,'' \emph{arXiv preprint arXiv:2102.01121}, 2021.
	
	\bibitem{azum1}
	B. Bercu, B. Delyon, E. Rio.  ``Concentration Inequalities for Sums and Martingales.,''\emph{Springer}, 2015.
	
	\bibitem{azum2}
	Lalley, S. P. ``Concentration inequalities,'' Lecture notes, University of Chicago. 2013.
	
	
	\bibitem{n1}
	B. Touri and A. Nedic,  ``Distributed consensus over network with noisy links,'' \emph{12th International Conference on Information Fusion, IEEE}, pp. 146-154, 2009.
	
	\bibitem{direct1}
	C. Xi and U. A. Khan,  ``Distributed subgradient projection algorithm over directed graphs,'' \emph{IEEE Trans. on Automat. Control},  vol. 62, no. 8, pp. 3986-3992, 2016.
	
	\bibitem{kc}
	K. Cai and H. Ishii,  ``Average consensus on general strongly connected digraphs,'' \emph{Automatica}, vol. 48, no. 11, pp. 2750-2761, 2012.
	
	
	\bibitem{n2}
	D. Wang, J. Zhou,  Z. Wang, and W. Wang,  ``Random gradient-free optimization for multiagent systems with communication noises under a time-varying weight balanced digraph,'' \emph{IEEE Transactions on Systems, Man, and Cybernetics: Systems},vol. 50, no. 1, pp. 281-289, 2017.
	
	\bibitem{ut}
	D. P. Palomar, M. Chiang,  ``Alternative distributed algorithms for network utility maximization: Framework and applications,'' \emph{IEEE Trans. on Automat. Control}, vol. 52, no. 12, pp. 2254-2269, 2017.
	
	\bibitem{sto}
	D. Yuan, D. W. C Ho, and Y. Hong,  ``On convergence rate of distributed stochastic gradient algorithm for convex optimization with inequality constraints,'' \emph{SIAM J. Control Optim.}, vol. 54, no. 5, pp. 2872-2892, 2016.
	
	\bibitem{yuanmd}
	D. Yuan, Y. Hong, D. W. C Ho, and S. Xu,  ``Distributed mirror descent for online composite optimization,'' \emph{IEEE Trans. on Automat. Control}, vol. 66, no. 2, pp. 714 - 729, 2021.
	
	\bibitem{dmcb1}
	D. Yuan, D. W. C Ho, G. P. Jiang. ``An adaptive primal-dual subgradient algorithm for online distributed constrained optimization,'' \emph{IEEE transactions on cybernetics}, vol. 48, no. 11, pp. 3045-3055, 2017.
	
	\bibitem{dmcb2}
	D. Yuan, D. W. C Ho, S. Xu. Regularized primal-dual subgradient method for distributed constrained optimization, \emph{IEEE transactions on cybernetics}, vol. 46, no. 9, pp. 2109-2118, 2015.
	
	\bibitem{micro}
	D. G. Luenberger,  ``Microeconomic theory,'' \emph{Mcgraw-Hill College}, (1995).
	
	
	\bibitem{noiau}
	H. Li, B. Jin, W. Yan. ``Distributed model predictive control for linear systems under communication noise: Algorithm, theory and implementation,'' \emph{Automatica}, vol. 125, 109422, 2021.
	
	\bibitem{ljy}
	J. Li, C. Gu, Z. Wu, T. Huang. ``Online Learning Algorithm for Distributed Convex Optimization With Time-Varying Coupled Constraints and Bandit Feedback,'' \emph{IEEE transactions on cybernetics}, DOI: 10.1109/TCYB.2020.2990796, 2020.
	
	
	
	\bibitem{duchi}
	J. C. Duchi, A. Alekh, and J. W. Martin,  ``Dual averaging for distributed optimization: Convergence analysis and network scaling,'' \emph{IEEE Trans. on Automat. Control}, vol. 57, no. 3, pp. 592-606, 2011.
	
	
	\bibitem{lmd}
	J. Li, G. Chen, Z. Dong, Z. Wu (2016). ``Distributed mirror descent method for multi-agent optimization with delay,'' \emph{Neurocomputing}, vol. 177, pp. 643-650.
	
	\bibitem{jq}
	J. Lu, D. Ho. ``Stabilization of complex dynamical networks with noise disturbance under performance constraint,'' \emph{Nonlinear Analysis: Real World Applications}, vol. 12, no. 4, pp. 1974-1984, 2011.
	
	\bibitem{zd1}
	J. Hu, Z. Wang, G. P Liu. ``Delay compensation-based state estimation for time-varying complex networks with incomplete observations and dynamical bias,'' \emph{IEEE Transactions on Cybernetics}, DOI: 10.1109/TCYB.2020.3043283, 2021.
	
	\bibitem{fl}
	L. Feng, L., J. Cao, L. Liu. ``Robust analysis of discrete time noises for stochastic systems and application in neural networks,'' \emph{International Journal of Control}, vol. 93 , no. 12, pp. 2908-2921, 2020.
	
	
	\bibitem{n3}
	K. Srivastava, A. Nedic, D. M. Stipanovic,  ``Distributed constrained optimization over noisy networks,'' \emph{In 49th IEEE Conference on Decision and Control (CDC) IEEE}, pp. 1945-1950, 2010.
	
	\bibitem{sensor}
	M. G. Rabbat, R. D. Nowak,  ``Decentralized source localization and tracking [wireless sensor networks],'' \emph{In  IEEE International Conference on Acoustics, Speech, and Signal Processing}, vol. 3, pp. iii-921, 2004.
	
	
	\bibitem{noiz}
	J. He, M. Zhou, P. Cheng, L. Shi, J. Chen. ``Consensus under bounded noise in discrete network systems: An algorithm with fast convergence and high accuracy,'' \emph{IEEE transactions on cybernetics}, vol. 46, no. 12, pp. 2874-2884, 2015.
	
	
	
	\bibitem{py}
	P. Yi, Y. Hong, and F. Liu,  ``Distributed gradient algorithm for constrained optimization with application to load sharing in power systems,'' \emph{Systems Control Lett.}, vol. 83, pp. 45-52,2015
	
	\bibitem{s2}
	S. S. Ram, A. Nedic, and V. V Veeravalli,  ``Distributed stochastic subgradient projection algorithms for convex optimization,'' Journal of Optim. Theory and Applications, vol. 147, no. 3, pp. 516-545, 2010.
	
	
	\bibitem{duallee}
	S. Lee, S, A. Nedic, M. Raginsky,  ``Stochastic dual averaging for decentralized online optimization on time-varying communication graphs,'' \emph{IEEE Trans. on Automat. Control}, vol. 62, no. 12,  pp. 6407-6414, 2017.
	
	\bibitem{duals}
	S. Liu, P. Y. Chen, and A. O. Hero,  ``Accelerated distributed dual averaging over evolving networks of growing connectivity,'' \emph{IEEE Trans. Signal Process.}, vol. 66, no. 7,  pp. 1845-1859, 2018.
	
	\bibitem{zd2}
	S. Liu, Z. Wang, B. Shen, G. Wei. ``Partial-neurons-based state estimation for delayed neural networks with state-dependent noises under redundant channels,'' \emph{ Information Sciences}, vol. 547, pp. 931-944, 2021.
	
	
	\bibitem{smart}
	T. H. Chang, and A. Nedic,  ``Distributed constrained optimization by consensus-based primal-dual perturbation method,'' \emph{IEEE Trans. on Automat. Control}, vol. 59, no. 6, pp. 1524-1538, 2014.
	
	
	\bibitem{discom}
	C. X. Shi, G. H. Yang, ``Distributed composite optimization over relay-assisted networks,'' \emph{IEEE Transactions on Systems, Man, and Cybernetics: Systems,} DOI: 10.1109/TSMC.2019.2963452, 2020.
	
	\bibitem{timev2}
	W. Zhang, P. Zhao, W. Zhu, S. C. Hoi, and T Zhang,  ``Projection-free distributed online learning in networks,''  \emph{Proceedings of the 34th International Conference on Machine Learning}, vol. 70, pp. 4054-4062, 2017.
	
	\bibitem{wj}
	J. Gao, P. Zhu, W. Xiong, J. Cao, L. Zhang, ``Asymptotic synchronization for stochastic memristor-based neural networks with noise disturbance,'' \emph{Journal of the Franklin Institute}, vol. 353, no. 13, 3271-3289, 2016.
	
	\bibitem{hong2}
	X. Zeng, P. Yi, Y. Hong, and L. Xie,  ``Distributed continuous-time algorithms for nonsmooth extended monotropic optimization problems,'' \emph{SIAM J.  Control Optim.}, vol. 56, no. 6, pp. 3973-3993, 2018.
	
	\bibitem{zd4}
	X. Wan, Z. Wang, M. Wu, X. Liu. ``State estimation for discrete time-delayed genetic regulatory networks with stochastic noises under the round-robin protocols,'' \emph{IEEE transactions on nanobioscience}, vol. 17, no. 2, pp. 145-154, 2018.
	
	\bibitem{ccc}
	Y. Tang, N. Li,  ``Distributed zero-order algorithms for nonconvex multi-agent optimization,'' \emph{57th Annual Allerton Conference on Communication, Control, and Computing (Allerton)}, pp. 781-786, 2019.
	
	\bibitem{hong}
	Y. Wang, W. Zhao, Y. Hong, and M. Zamani,  ``Distributed subgradient-free stochastic optimization algorithm for nonsmooth convex functions over time-varying networks,'' \emph{SIAM J. Control  Optim.}, vol. 57, no. 4,  pp. 2821-2842, 2019.
	
	\bibitem{zd3}
	Y. Yuan, Z. Wang, P. Zhang, H. Dong,  ``Nonfragile near-optimal control of stochastic time-varying multiagent systems with control-and state-dependent noises,'' \emph{IEEE transactions on cybernetics}, vol. 49, no. 7, pp. 2605-2617, 2018.
	
	\bibitem{noiac}
	N. Chatzipanagiotis, M. M. Zavlanos. ``A distributed algorithm for convex constrained optimization under noise,'' \emph{IEEE Transactions on Automatic Control}, vol. 61, no. 9, pp. 2496-2511, 2015.
	
	
	\bibitem{yzhang}
	Y. Zhang, Y. Lou, Y. Hong and L. Xie,  ``Distributed projection-based algorithms for source localization in wireless sensor networks,'' \emph{IEEE Trans. Wireless Communications}, vol. 14, no. 6 pp. 3131-3142, 2015.
	
	
	
	
	
	
	
	
	
	
	
	
\end{thebibliography}
\end{document}